\documentclass[reqno,microtype]{gtpart}
\usepackage{amssymb,latexsym}
\usepackage[utf8]{inputenc}
\usepackage[english]{babel}
\usepackage{amsmath,amsthm}
\usepackage{enumerate}
\usepackage{enumitem}
\usepackage{graphicx}
\usepackage{tcolorbox}
\usepackage{blindtext}
\usepackage{framed}
\usepackage{url}

\newcommand{\QQ}{\mathbb{Q}}
\newcommand{\ZZ}{\mathbb{Z}}

\newtcolorbox{mybox}[3][]
{
	colframe = #2!25,
	colback  = #2!10,
	coltitle = #2!20!black,  
	title    = {#3},
	#1,
}

\theoremstyle{theorem}
\newtheorem{theorem}{Theorem}
 
\theoremstyle{lemma}
\newtheorem{lemma}{Lemma}

\theoremstyle{definition}
\newtheorem{definition}{Definition}
\newtheorem*{remark}{Remark}

\makeatletter

\doubletitle [When do we have 1 + 1 = 11 and  2 + 2 =5? ] {When do we have 1 + 1 = 11 and  2 + 2 =5? \\ \\ ``\small Freedom is slavery. Two and two make five.'' - 1984.}
\author{R. Padmanabhan}
\address{
	Department of Mathematics\\
	University of Manitoba\\
	Winnipeg, Manitoba  R3T 2N2\\
	Canada}
\email{padman@cc.umanitoba.ca}

\author{Alok Shukla}
\address{
	Mathematical and Physical Sciences division\\
	Ahmedabad University\\
	India}
\email{Alok.Shukla@ahduni.edu.in}



\date{}

\begin{document}

\begin{abstract}
%

The phrase ``$ 2+2=5 $'' is a cliché or a slogan used in political speeches, propaganda, or literature, most notably in the novel ``$ 1984 $'' by George Orwell. More recently, we came across a You-Tube short film comedy, Alternative Math, produced by IdeaMan Studios (see \cite{Danny}). It is a hilarious exaggeration of a teacher who is dragged through the mud for teaching that $ 2+2=4 $ and not $ 22 $, as Danny, a young student, kept on insisting.  In the movie, Danny and the whole community sincerely believe $ 1+1=11 $ and $ 2+2=22 $. Jokes aside, we ask the question whether a polynomially defined group law ``$\oplus$'' defined over the field of rationals such that $ 1\oplus1=u $ and $ 2 \oplus2=v $ can simultaneously be satisfied for arbitrary integers $u$ and $ v $. Answer to this question takes us through a fascinating journey from Brahmagupta all the way to the modern works of Louis Joel Mordell and Ramanujan!

\end{abstract}
\maketitle
\section{Introduction.} \label{Sec:Intro}

In the standard mathematical world of arithmetic, the truth of `two plus two equals four' is self-evident. However, a protagonist in the `Notes from the Underground', by Fyodor Dostoevsky says: 

\textit{``I admit that twice two makes four is an excellent thing, but if we are to give everything its due, twice two makes five is sometimes a very charming thing too.''}

Taking inspiration from the works of fiction by Orwell and Dostoevsky, we ask rather seriously: can we really have $ 2 + 2 =5 $? If yes, then what about $ 1+1 $ in such a mathematical system? And, what if we desire both $ 1 +1 =11 $ and $ 2+2 = 5 $ to hold together? \footnote{In 2017, while election campaigning in the state of Gujarat, India, the Indian Prime Minister Narendra Modi declared, "1+1 is not 2 but 11 and together we will take Gujarat to new heights". Another political leader, the general secretary  of Communist Party of India (Marxist), Mr.~Sitaram Yechury said in 2017: ``Politics is not just arithmetic. Two plus two could become twenty two if we strengthened people’s struggle.'' See also a funny math video \cite{Danny}.}  

Here we ask a more general question: when does the group law ``$ \oplus $'' satisfy both $ 1\oplus1=u $ and $ 2\oplus2=v $?
To investigate this, let us fix a field $ K $. The case of interest for us will be when $ K $  is a finite extension of $ \QQ $. We will mainly consider the case of trivial extension, i.e., when $ K $ is $ \QQ $. We further impose the condition that $ \oplus  $ is defined by a polynomial $ P(x,y) \in K[x,y] $, i.e., $ x \oplus y = P(x,y) $.
Our main goal will be to answer the following two questions.
\begin{enumerate} 
	\item Given $ u $ and $ v $, find $ P $ and $ K $ such that $ 1 \oplus 1 = u $  and $ 2 \oplus 2 =v $?
	\item Given some conditions on $ u $ and $ v $, and $ K= \QQ $, will it be possible to have $ 1 \oplus 1 = u $  and $ 2 \oplus 2 =v $? 
\end{enumerate}
\begin{definition} \label{Def:main}
	Let $ K $ be a finite extension of $ \QQ $.	We say that \textbf{$\bf (u,v,P,K) $ is true}, if there exists a commutative group operation $ x \oplus  y $  given by a polynomial $ P(x ,y) \in K[x,y] $, such that
	\begin{align} \label{Eq:Defmainaxiom}
	1 \oplus 1 = u \qquad \text{and} \qquad 2 \oplus 2 = v.
	\end{align}
\end{definition}

\section{Characterization of the polynomial group operation \\$ \bf x  \oplus y = P(x,y) $.}

In order to address the questions raised in the previous section, we will first characterize the polynomial $ P(x,y) $, using the commutative and associative properties of the group law $ \oplus $. 
Suppose the highest degree of $ x $ in $ P(x,y) $ is $ n $. Then, associativity of $ P(x,y) $ implies $ P(P(x,y),z) = P(x,P(y,z)) $. The highest degree of $ x $ on the left is $ n^2 $, whereas it is $ n $ on the right side. Therefore, $ n^2=n $. We ignore the case $ n=0 $, as it will mean that $ P(x,y) $ doesn't depend on $ x $. A contradiction, since the group operation $ \oplus $ must depend on both $ x $ and $ y $. Similarly, the highest degree of $ y $ in $ P(x,y) $ is also $ 1 $. Next, commutativity of $ P(x,y) $ implies that
$ P(x,y) $ is symmetric in $ x $ and $ y $, hence it can be assumed that
\begin{align} \label{Eq:defPfirst}
P(x,y) = a xy + bx +by + c.
\end{align}
Using associativity, we write $ (1 \oplus 2) \oplus 3 = 1 \oplus (2 \oplus 3) $. Now
\begin{align}
&P(P(1,2),3) = P(1,P(2,3)) \nonumber \\
& \implies (2a + 3b + c,3) = P(1,6a+5b+c) \nonumber \\
&\implies  3a(2a+3b+c) + b(2a +3b+c +3) + c \nonumber \\
& \hspace{4cm}= a(6a+5b+c)+b(6a+5b+c+1)+c \nonumber \\ 
& \implies ac = b^2 -b.
\end{align}
If $ a=0 $, then $ b$ must be $ 1 $ as both $ a $ and $ b $ can't be zero simultaneously. Therefore, there are two possibilities for $ P(x,y) $.
Either,
\begin{align} \label{Eq:DefPone}
P(x,y) = x + y + c,
\end{align}
or else, if $ a \neq 0 $ then
\begin{align}\label{Eq:DefPtwo}
P(x,y) = axy +bx + by + \frac{b^2-b}{a}.
\end{align}
\begin{remark}
	See the proof of Lemma 5, \cite{brawley2001associative}, for a similar proof of the above result.
\end{remark}

The former case, $ P(x,y) = x+y +c $ is easy to deal with, and we immediately dispose it of. It is clear that in this case $ 1 \oplus 1 = u  $ and $ 2 \oplus 2 =v $ implies that
$ 2 + c =u $ and $ 4 + c = v $. But then $ v-u  = 2$. Therefore, we conclude that:
\begin{lemma}\label{Lemma:one}
	 If $ u,v \in \QQ $ be such that $  v -u =2$, then there exists a group operation $ \oplus $, given by the polynomial $ x \oplus y = P(x,y) = x+y +u-2 $, such that $ 1 \oplus 1 =u $ and $ 2 \oplus 2 =v $, i.e.,
	 $ (u,u+2,x+y+u-2,\QQ) $ is true.
\end{lemma}

\begin{remark}
	We note here that $ x \oplus y = P(x,y) $ given by Eq.~$ \eqref{Eq:DefPone} $ and Eq.~$ \eqref{Eq:DefPtwo} $, both define a group. It can directly be checked that $ -c $ is the identity element in the former case and $ \frac{1-b}{a} $ is the identity elements in the later case. Moreover, it is interesting to note that the group obtained is isomorphic to $ K $ in the former case, whereas the group obtained in the later case is isomorphic to the multiplicative group of the non-zero elements of the field, i.e., $ K^{\times} $. The isomorphism in the case when $ P(x,y) $ is defined by Eq.~$ \eqref{Eq:DefPtwo} $, is given by $ f(x) = ax + b $ as shown below.
	\begin{align*}
	f(P(x, y))  &	= aP(x,y) + b \\
&	= a(axy+bx+by+(b^2 - b)/a) + b \\
&	= a^2xy+abx+aby+(b^2 - b) + b \\
&	= (ax + b)(ay + b) \\
&	= f(x) \times f(y).	
	\end{align*}
This proves the isomorphism. In particular, to find the identity element e, we simply solve the equation $ f(e) = 1:  a e + b = 1 $, or $ e = (1-b)/a $. 
Moreover, the inverse of $0 \in K $ under the map $ f $ is $ -b/a $, therefore, the group given by $ x \oplus y = P(x,y) = axy+bx+by+(b^2 - b)/a$ is defined on the set $ K - \{-b/a\} $.	

In the following, we will continue to use the statement `$ (u,v,P,K) $ is true' as defined in Def.~$ \ref{Def:main} $, even when the group might be defined on the set $ K - \{-b/a\} $, rather than on $ K $.
\end{remark}

\section{Connections with number theory.}
Now we consider the other more interesting case $ a \neq 0 $, and assume that $ P(x,y) $ is given by Eq.~$ \eqref{Eq:DefPtwo} $.
Since, from Eq.~$ \eqref{Eq:Defmainaxiom} $, we have  $ 1 \oplus 1 = u $, and $ 2 \oplus 2 = v $, we need to solve for $ a,b $, in the following equations:
\begin{align}
a^2 - b + 2 a b + b^2 &= a u, \label{Eq:first}\\ 4 a^2 - b + 4 a b + b^2 &= a v. \label{Eq:second}
\end{align}
Subtracting Eq.~$ \eqref{Eq:first} $ from Eq.~$ \eqref{Eq:second} $ gives,
\begin{equation}\label{Eq:subtract}
3 a^2 + 2 ab = a (v-u)
\end{equation}	
Since $ a \neq 0 $ by assumption,  Eq.~$ \eqref{Eq:subtract} $ implies $  3 a + 2 b = v -u $, or $ b =  \frac{1}{2}(-3 a - u + v) $.  Substituting 
$ b =  \frac{1}{2}(-3 a - u + v) $ in  Eq.~$ \eqref{Eq:first} $ leads to a quadratic equation in $ a $, which can easily be solved.
We note the final solutions \footnote{One can use the Mathematica code $ solve (resultant (a^2 - b + 2ab + b^2 - au, 4a^2 - b + 4ab + b^2 - av, b), a) $ to get $ a $ as given in Eq.~$ \eqref{Eq:ab} $.}
\begin{align}\label{Eq:ab}
a = -3 + u + v \pm \sqrt{9 - 8 u - 4 v + 4 u v} & \quad \text{and} \qquad b =  \frac{1}{2}(-3 a - u + v).
\end{align}

For $ u=11 $ and $ v=5 $, from the above equation, we see that the solution is possible in $ \QQ $, i.e., $ (11,5,24xy - 39x - 39y + 65,\QQ) $ is true. As remarked earlier, we note that $ -b/a = 39/24 $ is excluded from $ \QQ $, while defining the group in this example. However, it is not always possible to have $ 1 \oplus 1 =u $ and $ 2 \oplus 2 =v $, defined by a polynomial group law $ \oplus $ over rationals, i.e., $ (u,v,P,\QQ) $ is not always true.
As an example, for $ u=11 $ and $ v=22 $ we have
\begin{align*}
a = 3( 10  \pm \sqrt{89}),\, b = \frac{1}{2}(11 - 3a).
\end{align*}
In this case, $ (11,22,P,\QQ) $ is false, but $ (11,22,P, \QQ(\sqrt{89})) $ is true, answering Danny's question, \cite{Danny}.

Now we consider, given $ u,v \in \ZZ $, whether $ (u,v,P,\QQ) $ is true (with $ P $ as defined in Eq.~$ \eqref{Eq:DefPtwo} $).
It is clear that for $ (u,v,P,\QQ) $ to be true $ 9 - 8 u - 4 v + 4 u v $ must be perfect square.
Consider the Diophantine equation  
\begin{align}
9 - 8 u - 4 v + 4 u v = n^2.
\end{align}
for all possible integers $ u,v $ and $ n $. 
First we assume that $ n $ is fixed. Clearly for a solution to exist $ n^2 \equiv 1 \mod 4  $. This means $ n $ must be odd. Then, in fact, $ n^2 \equiv 1 \mod 8  $. Let this be the case. Assume $ n= 2 m +1 $. Then, we have
\begin{align}
&2 - 2 u -  v + u v =  m(m+1)    \nonumber\\ 
\implies &(u-1)(v-2) = m(m+1) = \frac{1}{4}(n^2-1). \label{eq:m_and_m_plus_one}
\end{align}
From Eq.~$ \ref{eq:m_and_m_plus_one} $, for a given $ n $ we can count the number of solutions $ (u,v) $, and it is given by $ \sigma_0(\frac{1}{4}(n^2-1)) = $ the number of divisors of $ \frac{1}{4}(n^2-1) $.

Next, we assume that $ u,v $ are given. From Eq.~$ \eqref{eq:m_and_m_plus_one} $, a necessary condition for $ (u,v,P,\QQ) $ to be true is

\begin{align}
 (u-1)(v -2)  \equiv 0 \mod 2.
\end{align}

Since, $ n $ is odd,  $ \frac{n^2-1}{4} \geq 0 $. This means $ (u-1)(v-2) \geq 0 $.

Now we will prove a number of results characterizing the truth of  $ (u,v,P,\QQ) $. 
\begin{theorem}
\
	\begin{enumerate}[label=(\roman*)]
 \item Let $ u-1 $ and $ v-2 $ be prime numbers. Then, 
	\begin{align}
	(u,v,P,\QQ) \, \, \text{is true (clarify  Def.~$ \ref{Def:main} $) }  \implies  u=v = 4  \text{ or }  u =3 , v = 5.
	\end{align}
	\item $ |u-v+1 | =1  \implies (u,v,P,\QQ) \, \, \text{is true } $.
	\end{enumerate}
\end{theorem}

\begin{proof}\
	\begin{enumerate}[label=(\roman*)]
	\item 	Since, $ 	(u,v,P,\QQ)  $ is true, from Eq.~$ \eqref{eq:m_and_m_plus_one} $,
		\begin{align}
		(u-1)(v-2) =  \left(\frac{n-1}{2}\right) \left(\frac{n+1}{2}\right). \label{eq:pellgcd}
		\end{align}
		Since, $ \gcd(\frac{n+1}{2},\frac{n-1}{2}) =1 $, we have the following cases:
		\begin{align*}
		\begin{cases}
		u-1 =  \frac{n+1}{2} \qquad &\text{and } v -2 =  \frac{n-1}{2}, \\
		u-1 =  \frac{n-1}{2} \qquad &\text{and } v -2 =  \frac{n+1}{2}.\\
		\end{cases}
		\end{align*}
		From the above, we get that $ \frac{n-1}{2} $ and $ \frac{n+1}{2} $ must be consecutive primes. This means $ \frac{n-1}{2} =2 $ and $ n=5 $. Then, either $ u= v = 4$ or $ u= 3 $ and $ v = 5$ and the proof is complete.\\
		\item 
		As $ |u-v+1 | =1 $, we get either $ u=v $ or $ u = v-2 $. If $ u=v $, then $ 9 - 8 u - 4 v + 4 u v = (2u-3)^2 $. Next, if $ u =v-2 $, then $ 9 - 8 u - 4 v + 4 u v = (2v-5)^2 $. Therefore, in both the cases $ (u,v,P,\QQ)  $ is true, and we are done.
		\end{enumerate}
\end{proof}

	Our next result is related to Pell's equation, i.e., equation of the form $ n^2 - d t^2 = 1 $, where $  d $ is a fixed positive non-square integer, and integer solutions for $ (n,t) $ are sought. In order ot solve Pell's equation, it is helpful to consider the factorization $  n^2 - d t^2 = (n+ \sqrt{d} t) (n-  \sqrt{d} t)  $. For simplicity assume $ d $ is square-free and also $ d \not \equiv 1  \mod 4$. Then numbers of form $n+\sqrt{d}t  $ form a ring $ \ZZ[\sqrt{d}] $, which is the ring of integers for the number field $ \QQ(\sqrt{d}) $. For an element $ z = n+\sqrt{d}t  $, we define its conjugate as $ \bar{z} =  n- \sqrt{d} t  $. Then norm is defined as $ N(z) = z \bar{z} $. It is easy to see that norm is multiplicative, i.e., $ N(ab) = N(a)N(b) $ for $ a,b \in {\ZZ[\sqrt{d}]} $. Any $ z = n+\sqrt{d}t $ is a solution to Pell's equation $ n^2 - d t^2 = 1 $, if $ N(z)=1 $.  Let $ \epsilon > 1 $ be the smallest number in $ \ZZ[\sqrt{d}]$, such that $ N(\epsilon)=1 $, then $ \epsilon $ is a fundamental unit. Indeed, every unit of norm $ 1 $ is of the form of $ \pm\epsilon^k $ for some integer $ k$. To see this suppose that $ N(z) = 1 $. If necessary replacing $ z $ by $ -z $ we can assume $ z > 0 $. There exist a unique integer $ k $ such that $ \epsilon^k  \leq z < \epsilon^{k+1} $. Then  $ N(\epsilon^{-k} z)= N(\epsilon)^{-k} N(z)=1 $ and $ 1 \leq \epsilon^{-k} z < \epsilon $. Thus the existence of $ \epsilon_1 = \epsilon^{-k} z $ contradicts the minimality of $ \epsilon $ unless $ \epsilon_1 =1 $. Assuming this to be the case, we obtain $ z = \epsilon^{k} $, as claimed. 
	
	Now we give our next result.
\begin{theorem} 
	If $ (u-1)(v-2) = 2 t^2  $, where $ u,v $ and $ t > 0$ are integers, then
\begin{align}
	(u,v,P,\QQ) \, \, \text{is true (clarify  Def.~$ \ref{Def:main} $) }  \iff 	t = \frac{(3+2 \sqrt{2})^m - (3 - 2 \sqrt{2})^m}{4 \sqrt{2}} & \label{Eq:texp}
	\\ \qquad \qquad \text{for some positive integer $ m $.}\nonumber
\end{align}	
\end{theorem}
\begin{proof}
Assume $ (u,v,P,\QQ)  $ to be true. Then from Eq.~$ \eqref{eq:m_and_m_plus_one} $ and the given hypothesis $ (u-1)(v-2) = 2 t^2 $, we obtain
	\begin{align}
	n^2 - 8 t^2 = 1. \label{eq:pelld2}
	\end{align}
	This is a special case of Pell's equation $ n^2 - d t^2 = 1 $, with $ d=8 $. In the number field $ K = \QQ(\sqrt{2}) $, we can write
	\begin{align*}
	(n+2 \sqrt{2} t) (n- 2 \sqrt{2} t) =1.
	\end{align*}
	
	Then $ N(n+ 2 \sqrt{2} t) = 1$, where $ N(\cdot) $ is the usual norm in  $ K$. We note that $ n=3,2t=2 $ is a solution of Eq.~$ \ref{eq:pelld2} $, and in fact, $ 3 + 2 \sqrt{2} $ is a fundamental unit (see Table 4, Page 280, \cite{alaca2004introductory}). Then from Theorem 11.3.2, \cite{alaca2004introductory}, it follows that any solution of $ n^2 - 8 t^2 =1  $, will be such that  $ n+ 2 \sqrt{2} t = \pm (3 +\sqrt{2} \cdot 2)^m $ for some  integer   $ m $. We can assume $ m $ to be a positive integer. Then on solving for $ t $, in  $ n+ 2\sqrt{2} t =  (3 +\sqrt{2} \cdot 2)^m $ and $ n-2 \sqrt{2} t =  (3 +\sqrt{2} \cdot 2)^{-m} =  (3 - \sqrt{2} \cdot 2)^{m} $, the result follows. 
	
	Now, to prove the other direction assume that $ t $ is given by Eq.~$ \eqref{Eq:texp} $. Then on defining 
	\begin{align}
	n = \frac{(3+2 \sqrt{2})^m + (3 - 2 \sqrt{2})^m}{2}, \label{Eq:nexp}
	\end{align}
	we see that $ n^2 - 8 t^2 =1 $. This along with $ (u-1)(v-2) =2 t^2$ gives $ (u-1)(v-2) = \frac{n^2-1}{4} $. Then,  $ 9 - 8 u - 4 v + 4 u v = n^2 $. Therefore, $ (u,v,P,\QQ) $ is true. This completes the proof for the other direction.
\end{proof}

We state the following result, whose proof is similar to the previous theorem. 
\begin{theorem}
	Let $ \alpha_d + \beta_d \sqrt{2d}  $ be the fundamental unit of the number field $ K = \QQ(\sqrt{2d}) $, where $ d $ is a square free odd integer.
	Also assume $ (u-1)(v-2) = 2 d t^2  $, with $ u,v $ and $ t > 0$ being integers. Then
	\begin{align}
	(u,v,P,\QQ) \, \, \text{is true (clarify  Def.~$ \ref{Def:main} $) }  \iff \nonumber \\	t  = \frac{(\alpha_d + \beta_d \sqrt{2d})^m - (\alpha_d - \beta_d \sqrt{2d})^m}{4 \sqrt{2d}}, & \label{Eq:texp2}
	\end{align}	
	for some positive integer $ m $. 
\end{theorem}
\begin{remark}[Chakravala: an ancient algorithm.]
We note that the proof of the above theorem depended upon the solution of Pell's equation $ n^2 -8dt^2 =1 $. Indeed, Pell's equation has a very interesting history. It is one of the cases of wrong attributions in mathematics. 	

In 1657 Fermat posed a challenge to mathematicians. The challenge was to find integer solutions for the equation $ x^2 - N y^2 = 1 $, for values of $N$ like $ N = 61, 109 $. 	Several centuries earlier, in 1150, Bhaskara II had already found solutions for the problem proposed by Fermat. 
	\begin{align*}
	1766319049^2  - 61(226153980)^2 &= 1 \\
	158070671986249^2 - 109(15140424455100)^2 &= 1.
	\end{align*}
 In fact, Brahmagupta (598-665) had already solved this equation in the early seventh century for various values of $ N $, such as $ N = 83 $ and $ N = 92 $.
Brahmagupta viewed these problems very highly and he had remarked: ``Any person who is able to solve these two cases, within a year, is truly a mathematician''!

Let $ (a, b; m) $ to denote an integer solution of the equation $ x^2 - Ny^2 = m $. 
Brahmagupta discovered the following `composition rule' in the early seventh century.

		\begin{align}
		(a, b; m) * (c, d; n) \to (ac \pm N bd, ad \pm bc; mn).
		\end{align}
This is perhaps one of the earliest example of the use a `group-theoretic' argument in mathematics.  This	`composition rule' allowed Brahmagupta to  obtain new solutions from old known solutions, since clearly by composing a known solution $ (a,b;m) $ with a triple $ (p, q; 1) $,  one can easily find  new solutions $ (ap \pm N bq, aq \pm bp; n) $.
Later, Jayadeva, Narayana and Bhaskara had refined and built on the works of Brahmagupta to devise an algorithm called the ``Chakravala'' for finding all the integer solutions of the equation $ x^2 -N y^2 = \pm 1 $ for any positive integer $ N $. We refer readers to \cite{weil2006number} for an interesting discussion on this. 
\end{remark}

\begin{theorem}
	If $ (u-1)(v-2) =  2 t^3  $, where $ u,v $ and $ t $ are integers, then
	\begin{align}
	(u,v,P,\QQ) \, \, \text{is true (clarify  Def.~$ \ref{Def:main} $) }  \iff 	t  = 0 \text{ or } t = 1.  \label{Eq:texp3}
	\end{align}	
\end{theorem}
\begin{proof}
	From $ (u-1)(v -2) = \frac{1}{4}(n^2 -1)$ we need to solve,
	\begin{align}
	n^2 =  (2t)^3 + 1. \label{eq:pelld2}
	\end{align}
	This is a special case of  Mordell’s equation $ y^2 = x^3 + k $, with $ k=1 $. 
It is known that only integral solutions of this equation are $ (x, y) = (-1, 0),(0, \pm 1)$, and $(2, \pm 3)$ (See Theorem 5, Chapter 26, Page 247, \cite{mordell1969diophantine}). 
\end{proof}

For example, we see that if $ t=0 $, then $ (u-1)(v-2) =0 $. Suppose, $ u=1 $. Then from Eq.~$ \eqref{Eq:ab} $ we get $ (a,b) = (v-1,1-v) $ or $ (a,b)=(v-3,4-v) $ and $ P(x,y) $ as given in Eq.~$ \eqref{Eq:DefPtwo} $ is a rational polynomial. Similarly, if $ v=2 $ then from Eq.~$ \eqref{Eq:ab} $ we get $ (a,b) = (u,1-2u) $ or $ (a,b)=(u-2,4-2u) $ and $ P(x,y) $ is again a rational polynomial in this case.

\begin{theorem}
	If $ (u-1)(v-2) =  2 (2^{t-3}-1)  $, where $ u,v $ and $ t $ are integers with $ t > 1 $, then
	\begin{align}
	(u,v,P,\QQ) \, \, \text{is true (clarify  Def.~$ \ref{Def:main} $) }  \iff 	t \in \{3,4,5,7,15\}.  \label{Eq:texp3}
	\end{align}	
\end{theorem}
\begin{proof}
	First let $ (u,v,P,\QQ)  $ be true. Then from Eq.~$ \eqref{eq:m_and_m_plus_one} $ and the given hypothesis $ (u-1)(v-2) = 2 (2^{t-3}-1)  $, we get
	\begin{align}
	n^2 +7 = 2^t. \label{eq:ramanujan}
	\end{align}
	We recall this is  Ramanujan-Nagell equation. 	In fact, Ramanujan conjectured in 1913 that $ (1, 3), (3, 4), (5, 5), (11, 7)$ and
	$(181, 15) $ are only positive solutions $ (n, t) $ of the Diophantine equation $ n^2 + 7 = 2^t $.
	Nagell proved this conjecture in $ 1948 $, \cite{nagell1961diophantine}. Therefore, from this our result follows.
\end{proof}
Finally, we would like to add yet another ``strange'' multiplication $ \oplus $ corresponding to the case of
$ t = 4 $ of the Ramanujan-Nagell equation. In this case $ u = 2 $, $ v = 4 $, and the group law is
\[
x \oplus y = 6xy - 8x - 8y + 12.
\]
Here we have $ 1 \oplus 1 = 2  $ and $ 2 \oplus 2 = 4  $, which looks like the ordinary addition, but it is a different group operation.
The element $ \frac{4}{3} $ is an ``annihilator''  of the group because:
\[
x \oplus \frac{4}{3} = 8x - 8x - 8 \times \left(\frac{4}{3}\right) + 12 =  \frac{4}{3}.
\]

\bibliographystyle{plain}

%
%

\vfill\eject

\end{document}